\documentclass{amsart}

\usepackage{amsmath,amssymb,yhmath, url,color}

\newtheorem{lemma}{Lemma}[section]
\newtheorem{prop}[lemma]{Proposition}
\newtheorem{cor}[lemma]{Corollary}
\newtheorem{thm}[lemma]{Theorem}
\newtheorem{example}[lemma]{Example}
\newtheorem{thm?}[lemma]{Theorem?}

\newtheorem{ques}[lemma]{Question}

\newtheorem{fact}{Fact}

\begin{document}
\title{The Euclidean Criterion for Irreducibles}
%\title{Euclid, Jacobson and Furstenberg: Squeezing Modern Mathematics From an Ancient Proof}
% When Furstenberg Domains Have Infinitely Many %Irreducibles}
%Jacobson-Semisimple Furstenberg Domains Have Infinitely Many Irreducibles...}

%Euclid, Jacobson and Furstenberg: Infinitely Many Irreducibles}
% Furstenberg and Jacobson: Squeezing Contemporary Mathematics From an Ancient Proof}

%Euclid, Jacobson and Furstenberg: Modern Mathematics Behind an Ancient Proof}
\author{Pete L. Clark}

\date{\today}

%\date{May 9th, 2015; small modifications on July 3, 2015}

%\address{Department of Mathematics \\ Boyd Graduate Studies Research Center \\ %University
%of Georgia \\ Athens, GA 30602-7403 \\ USA}
%\email{pete@math.uga.edu}
%\email{bcook@math.ubc.ca}
%\email{stankewicz@gmail.com}

%

\newcommand{\etalchar}[1]{$^{#1}$}
\newcommand{\F}{\mathbb{F}}
\newcommand{\et}{\textrm{\'et}}
\newcommand{\ra}{\ensuremath{\rightarrow}}
\newcommand{\FF}{\F}
\newcommand{\Z}{\mathbb{Z}}
\newcommand{\N}{\mathcal{N}}
\newcommand{\ch}{}
\newcommand{\R}{\mathbb{R}}
\newcommand{\PP}{\mathbb{P}}
\newcommand{\pp}{\mathfrak{p}}
\newcommand{\C}{\mathbb{C}}
\newcommand{\Q}{\mathbb{Q}}
\newcommand{\tpqr}{\widetilde{\triangle(p,q,r)}}
\newcommand{\ab}{\operatorname{ab}}
\newcommand{\Aut}{\operatorname{Aut}}
\newcommand{\gk}{\mathfrak{g}_K}
\newcommand{\gq}{\mathfrak{g}_{\Q}}
\newcommand{\OQ}{\overline{\Q}}
\newcommand{\Out}{\operatorname{Out}}
\newcommand{\End}{\operatorname{End}}
\newcommand{\Gon}{\operatorname{Gon}}
\newcommand{\Gal}{\operatorname{Gal}}
\newcommand{\CT}{(\mathcal{C},\mathcal{T})}
\newcommand{\ttop}{\operatorname{top}}
\newcommand{\lcm}{\operatorname{lcm}}
\newcommand{\Div}{\operatorname{Div}}
\newcommand{\OO}{\mathcal{O}}
\newcommand{\rank}{\operatorname{rank}}
\newcommand{\tors}{\operatorname{tors}}
\newcommand{\IM}{\operatorname{IM}}
\newcommand{\CM}{\operatorname{CM}}
\newcommand{\Frac}{\operatorname{Frac}}
\newcommand{\Pic}{\operatorname{Pic}}
\newcommand{\coker}{\operatorname{coker}}
\newcommand{\Cl}{\operatorname{Cl}}
\newcommand{\loc}{\operatorname{loc}}
\newcommand{\GL}{\operatorname{GL}}
\newcommand{\PSL}{\operatorname{PSL}}
\newcommand{\Frob}{\operatorname{Frob}}
\newcommand{\Hom}{\operatorname{Hom}}
\newcommand{\Coker}{\operatorname{\coker}}
\newcommand{\Ker}{\ker}
\renewcommand{\gg}{\mathfrak{g}}
\newcommand{\sep}{\operatorname{sep}}
\newcommand{\new}{\operatorname{new}}
\newcommand{\Ok}{\mathcal{O}_K}
\newcommand{\ord}{\operatorname{ord}}
\newcommand{\mm}{\mathfrak{m}}
\newcommand{\Ohell}{\OO_{p^{\infty}}}
\newcommand{\ff}{\mathfrak{f}}
\renewcommand{\N}{\mathbb{N}}
\newcommand{\Gm}{\mathbb{G}_m}
\newcommand{\MaxSpec}{\operatorname{MaxSpec}}
\newcommand{\qq}{\mathfrak{q}}
\newcommand{\cof}{\operatorname{cf}}
\newcommand{\lub}{\operatorname{lub}}
\newcommand{\glb}{\operatorname{glb}}
\newcommand{\Hol}{\operatorname{Hol}}

\begin{abstract}
We recast Euclid's proof of the infinitude of prime numbers as a \textbf{Euclidean Criterion} for a domain to have infinitely many atoms.  We make connections with Furstenberg's \emph{``topological'' proof} of the infinitude of prime numbers and show that our criterion applies even in certain domains in which not all nonzero nonunits factor into products of irreducibles.
\end{abstract}

\maketitle

%\tableofcontents

\noindent
%Dear Math Research VRGers,
%\\ \\
%I wanted to say a bit more about the first proof of Euclid's Theorem that I presented in the first meeting of the VRG on August 20, 2015.  I %fear that it may have looked like I was making a joke by putting a transparent algebraic veneer on Euclid's proof.  There was some humor in %it, but I had intended more than that.
%\\ \\
%In particular, though I perhaps did not pull it off in the best possible way and am worried that I may have immediately violated the precept %I set down that concepts used in proofs should be accessible to the entire group, I did have in mind the exhibition of a certain technique %in mathematical research, namely \textbf{axiomatization of a proof}.
%\\ \\
%So let me say a bit more about that.
%\\ \\
%\textbf{Added at the end:} This document turned out to be rather longer than I had originally intended.  The decision to show it to you was %made in consultation with Paul Pollack.  We agreed that it should be made clear that this is \textbf{optional reading}.  One fairly evident %but very important part of choosing your own research path: it is far better to explore things that are actually of interest to you!  If %this stuff is not to your taste, no worries.

\section{Introduction}
\noindent
This article has its genesis in a graduate \emph{VIGRE research group} taught by Paul Pollack and me in Fall 2015: \emph{Introduction to the Process of Mathematical Research}.  Rather than concentrating on a fixed topic preselected by us, the goal was to guide students through the process of selecting and performing research on their own.  One technique we tried to inculcate is exploitation of the many-to-one relation between theorems and proofs.  A good theorem has several proofs, and you will know two proofs are different when can be used to prove further theorems the other cannot.  \\ \indent In our first meeting, Pollack and I presented seven proofs of Euclid's Proposition IX.20: there are infinitely many prime numbers.  My first proof: suppose given a domain $R$ that is not a field, in which each nonzero nonunit factors into irreducibles and whenever $x \in R$ is a nonzero nonunit then $x+1$ is not a unit; then there is at least one irreducible element $f_1$, and given irreducibles $f_1,\ldots,f_n$, by factoring $f_1 \cdots f_n + 1$ we get a new irreducible element.  It was pointed out that this argument, though correct, does not imply Euclid's result: $x = -2$ is a problem.  
Some salvages were suggested: in $\Z$ it is enough to replace $f_1 \cdots f_n$ by $- f_1 \cdots f_n$, if necessary.
%\\ \\
%One more research technique: after proving a theorem, %look back at the hypotheses.  Did you use all of %them?  If not, you've proved a stronger theorem.  If %so, did you really need to -- does the result become %false if any hypothesis is removed?
\\ \indent
Here we present a general fix  -- a \textbf{Euclidean Criterion} for a domain to have infinitely many nonassociate irreducibles -- and explore its consequences.  We soon find ourselves on a scenic tour of 20th century mathematics, as we engage with work of Jacobson, Furstenberg, Cohen-Kaplansky and Anderson-Mott, among others.

\subsection{Acknowledgments}
\textbf{} \\ \\ \noindent
Thanks to all members of the 2015-2016 Introduction to Mathematical Research UGA VIGRE group.  Conversations with Saurabh Gosavi, Noah Lebowitz-Lockard, Robert Samalis, Lee Troupe and Lori D. Watson were helpful.  \\ \indent
My group coleader Paul Pollack made key contributions: first, he emphasized that the Euclidean Criterion automatically yields pairwise comaximality.  Second, Theorem \ref{POLLACKTHM} was inspired by \cite[Thm. 1.16]{Pollack}, and though I came up with the statement, I could prove it only in various special cases.  The proof included here is his. \\ \indent
I am grateful to two anonymous referees for their careful, detail-oriented reports.  In particular, Example \ref{EXAMPLE4.19} was 
suggested by the ``first'' referee.

\section{The Euclidean Criterion}

\subsection{A primer on factorization in domains}
\textbf{} \\ \\
%The basic terminology introduced here should be familiar from undergraduate abstract algebra...but is rarely taught in the detail that it %deserves.  For a more leisurely treatment of these topics, see \cite{Clark-F}.
%\\ \\
By a \textbf{ring} we will mean a commutative ring with a multiplicative identity.  We denote the set of nonzero elements of $R$ by $R^{\bullet}$.  An element $x \in R$ is a \textbf{unit} if there is $y \in R$ such that $xy = 1$.  We denote the group of units of $R$ by $R^{\times}$.  For a subset $S$ of a ring $R$, we denote by $(S)$ the ideal of $R$ generated by $S$.  (As is standard, we write $(x_1,\ldots,x_n)$ for $(\{x_1,\ldots,x_n\})$. Ideals $I$ and $J$ in $R$ are \textbf{comaximal} if $I+J = R$.  Elements $a,b \in R$ are comaximal if $(a)$ and $(b)$ are comaximal: $(a,b) = R$.  An indexed family of ideals $\{I_i\}$ is \textbf{pairwise comaximal} if $I_i + I_j = R$ for all $i \neq j$, and similarly for pairwise comaximal elements.
\\ \indent
A \textbf{domain} is a nonzero ring in which $x,y \neq 0 \implies xy \neq 0$. For $x,y \in R$ we say \textbf{x divides y} and write $x \mid y$ if there is $c \in R$ such that $cx = y$.  Elements $x$ and $y$ are \textbf{associates} if $y = ux$ for some $u \in R^{\times}$.  An element $x$ of a domain is \textbf{irreducible} if it is a nonzero nonunit and $x = yz$ implies $y \in R^{\times}$ or $z \in R^{\times}$.  A \textbf{prime element} $p \in R$ is an element $p \in R^{\bullet}$ for which $(p)$ is a prime ideal.  Thus a nonzero nonunit $p$ is prime if and only if $p \mid ab \implies p \mid a$ or $p \mid b$. 
\\ \indent
An \textbf{atom} in a domain $R$ is a principal ideal $(x)$ generated by an irreducible element $x$.  Thus two irreducibles of a domain $R$ determine the same atom if and only if they are associate.  (It is more common in the literature for the terms ``atom'' and 
``irreducible'' to be fully synonymous, but this minor distinction is convenient for our purposes: usually we will count to count irreducibles in a domain \emph{up to associates}, but sometimes we will want to count irreducibles.) 
A \textbf{Furstenberg domain} is a domain $R$ in which every nonzero nonunit has an irreducible divisor.\footnote{The explanation for the terminology comes in $\S$3.1.} An \textbf{atomic domain} is a domain $R$ in which for every nonzero nonunit 
$x \in R$ there are irreducible elements $f_1,\ldots,f_n$ such that $x = f_1 \cdots f_n$.  A \textbf{unique factorization domain (UFD)} is an atomic domain such that if $f_1,\ldots,f_m,g_1,\ldots,g_n$ are irreducibles such that $f_1 \cdots f_m = g_1 \cdots g_n$, then
 $m = n$ and there is a bijection $\sigma: \{1,\ldots,m\} \ra \{1,\ldots,n\}$ such that $(f_i) = (g_{\sigma_i})$ for all $1 \leq i \leq m$.  
\\ \indent
Prime elements are irreducible.  In general the converse is false!  An atomic domain is a UFD iff every irreducible is prime 
%\cite{Cohn73}, 
\cite[Thm. 15.8]{Clark-CA}. The terminology can be confusing in light of the definition of a prime number $p$ as a positive integer not divisible by any $1 < n < p$: this means $p$ is irreducible in $\Z$.  But Euclid showed
\[ p \mid ab \implies p \mid a \text{ or } p \mid b. \]
From this one can easily show the Fundamental Theorem of Arithmetic: $\Z$ is a UFD.  
%This is an immediate consequence of unique %factorization in $\Z$: the Fundamental Theorem of %Arithmetic.  But a domain is a unique factorization %domain precisely when it is a factorization domain %(which for $\Z$ is clear: apply well-ordering) in %which every irreducible element is prime.  So while %Euclid did not state the Fundamental Theorem of %Arithmetic, he stated and proved a result which %carries all its content. \\ \\
A \textbf{principal ideal domain (PID)} is a domain in which each ideal is generated by a single element.  Every PID is a UFD.  It follows from the Euclidean algorithm that  $\Z$ is a PID.   A \textbf{B\'ezout domain} is a domain in which every finitely 
generated ideal is principal. 
A ring is \textbf{Noetherian} if all of its ideals are finitely 
generated.  Noetherian domains are atomic \cite[Prop. 15.3]{Clark-CA}.  Thus a PID is precisely a Noetherian B\'ezout domain.  A \textbf{Dedekind domain} is a domain in which each nonzero proper ideal factors uniquely into prime ideals.  A domain is Dedekind iff it is Noetherian, \textbf{of dimension at most one} -- every nonzero prime ideal is maximal -- and \textbf{integrally closed} -- every element of the fraction field which satisfies a monic polynomial with coefficients in $R$ lies in $R$ \cite[Thm. 20.10]{Clark-CA}.
\\ \indent
Working in a \emph{domain} rather than a general ring confers certain advantages:
 % Some very simple -- but useful -- consequences are recorded in the following result.

\begin{fact}
\label{FACT}
a) Every nonzero ideal in a ring contains a nonzero principal ideal. \\
b) If $R$ is a domain and $\alpha \in R^{\bullet}$, $x \in R \mapsto \alpha x$ gives a bijection from $R$ to $(\alpha)$. \\
% $R \ra (\alpha)$.  \\
c) Thus for every nonzero ideal $I$ of a domain $R$ we have $\# I = \# R$. \\
d) For nonzero ideals $I$ and $J$ of $R$, $I \cap J$ contains $IJ$ and thus is nonzero.
\end{fact}

%Notice that a unit is comaximal to anything, and two %comaximal nonunits are necessarily nonzero and %nonassociate.

%\\ \\
%So a question occurs: in seeking generalizations of Euclid's Theorem to other domains, should we be trying to produce irreducible elements %or prime ideals?  In fact both generalizations are possible.

\subsection{The Euclidean Criterion}
\textbf{} \\ \\ \noindent
A ring $R$ satisfies \textbf{Condition (E)} if for all $x \in R^{\bullet}$, there is $y \in R$ such that $yx + 1 \notin R^{\times}$.  In other words, if $x \neq 0$ then $1+(x) \not \subset R^{\times}$.  By Fact \ref{FACT}a) this is equivalent to: if $I$ is a nonzero ideal of $R$ then $1 + I \not \subset R^{\times}$, though we will defer consideration of this restatement until later on.

\begin{example}
\label{1.1} \textbf{} \\
a) The ring $\Z$ satisfies Condition (E).  Indeed, $\Z^{\times} = \{\pm 1 \}$, so for $x \in \Z^{\bullet}$,
take $y = 1$ if $x$ is positive and $y = -1$ if $x$ is negative); then $yx \geq 1$ so $yx + 1 \geq 2$.  \\
b) For any domain $R$, the polynomial ring $R[t]$ satisfies Condition (E).  Indeed, $(R[t])^{\times} = R^{\times}$, so for any $x \in R[t]^{\bullet}$, take $y = t$.  \\
c) $R = \Z[i]$ satisfies Condition (E).  Indeed $\Z[i]^{\times} = \{1,i,-1,-i\}$, so this is geometrically clear: for any $x \in \Z[i]^{\bullet}$, if we multiply it by a $y$ with large enough $|y|$, then $yx$ will be much more than $1$ unit away from any point on the unit circle.
% \\
%d)  
\end{example}
%\noindent
%Let $K$ be a number field.  Does the ring of integers $\Z_K$ satisfy Condition (E)?  (The answer will come later, but I invite you to think %about it now.)

\begin{prop}
\label{1.4}
A domain $R$ with $\# R > \# R^{\times}$ satisfies Condition (E).
%Let $R$ be a domain.  If $\# R > \# R^{\times}$,
%$R$ satisfies Condition (E).
\end{prop}
\begin{proof}
For $x \in R^{\bullet}$, the map $\iota: R \ra R$ given by
$y \mapsto yx + 1$ is an injection.
% indeed, if $y_1x + 1 = y_2 x + 1$, then $(y_1-y_2)x = 0$, and since $R$ is a domain, $y_1 = y_2$.
  Thus $\# \iota(R) = \# R > \# R^{\times}$, so it cannot be that $\iota(R) \subset R^{\times}$.
\end{proof}
\noindent
And here we go:

%\begin{mainthm}
\begin{thm}(The Euclidean Criterion) 
\label{BIGONE}
\\
Let $R$ be a domain, not a field, satisfying Condition (E).  \\
a) There is an infinite sequence $\{a_n\}_{n=1}^{\infty}$ of pairwise comaximal nonunits.  \\
b) If $R$ is also Furstenberg, it admits an infinite sequence $\{f_n\}_{n=1}^{\infty}$ of pairwise comaximal irreducibles.  Thus $\{ (f_n)\}_{n=1}^{\infty}$ is a sequence of 
distinct atoms in $R$.
\end{thm}
%\end{mainthm}
\begin{proof}
a) By induction on $n$.  Let $a_1 \in R$ be a nonzero nonunit. Having chosen $a_1,\ldots,a_n$ pairwise comaximal, by Condition (E) there is $y \in R$ such that $a_{n+1} := y a_1 \cdots a_n + 1 \notin R^{\times}$.  Clearly $(a_i, a_{n+1}) = R$ for all $1 \leq i \leq n$. \\
b) By induction on $n$.  Since $R$ is Furstenberg and not a field, it has an irreducible $f_1$.  Having chosen pairwise comaximal irreducibles $f_1,\ldots,f_n$, by Condition (E) there is $y \in R$ such that $x = y f_1 \cdots f_n + 1$ is a nonzero (since $f_1 \notin R^{\times}$)  nonunit, so $x$ has an irreducible factor $f_{n+1}$.  For all $1 \leq i \leq n$ we have \[1 = (x/f_{n+1}) f_{n+1} -(y \prod_{j \neq i} f_j)f_i, \] so $f_i,f_{n+1}$ are comaximal.  Finally, if $x$ and $y$ are pairwise comaximal irreducibles, then $(x),(y) \subsetneq R$ 
and $(x) + (y) = (x,y) = R$, so we must have $(x) \neq (y)$.
\end{proof}
\noindent
Here are two applications of the Euclidean Criterion.  The first two are immediate.

\begin{thm}
\label{2.1}
a) For any domain $R$, $R[t]$ has infinitely many atoms.  \\
b) In particular, let $D$ be a UFD and let $R = D[t_1,\ldots,t_n]$.  Then $R$ is a UFD satisfying Condition (E), so $R$ has infinitely many nonassociate prime elements.  \\
c) The Gaussian integers $\Z[i]$ have infinitely many atoms.  Since $\Z[i]$ is a PID, there are infinitely many nonassociate prime elements.
\end{thm}

\begin{thm}
\label{2.2}
Let $R$ be a Furstenberg domain, not a field, such that $\# R > \# R^{\times}$.  Then $R$ has infinitely many atoms.
%In particular, an infinite PID with finite unit group %%has infinitely many prime elements iff it has at %least one.
\end{thm}

\begin{thm}
\label{2.3}
Let $R$ be a Furstenberg domain, let $\mathcal{I}$ be the set of all irreducible elements of $R$.  Then $\mathcal{I}$ is either empty (if $R$ is a field) or infinite (otherwise).
\end{thm}
\begin{proof}
Assume $\mathcal{I} \neq \varnothing$ and fix $f \in \mathcal{I}$.  If $R^{\times}$ is finite, Theorem \ref{2.2} yields infinitely many atoms.  If $R^{\times}$ is infinite, then $\{uf \mid u \in R^{\times}\}$ is an infinite subset of $\mathcal{I}$.
\end{proof}

\subsection{Supplement: Irreducibles in Residue Classes}
\textbf{} \\ \\ \noindent
We switch from an ancient theorem to matters of contemporary interest if we ask for infinitely many primes \emph{satisfying certain additional conditions}.  Here is a result along these lines,  relatively modest over $\Z$, but of a general algebraic nature.

% not so impressive in the cases of classical interest but satisfying in its generality.

\begin{lemma}
\label{LITTLECRTLEMMA}
Let $a,b,c$ be elements of a ring $R$.
%Let $a,b,c,d_1,\ldots,d_n$ be elements in a ring $R$. %\\
If $(a,b) = R$ and $c \mid a+b$, then $(a,c) = (b,c) = R$.
%\\
%b) If $(a,b_i)= R$ for all $1 \leq i %\leq n$, then $(a,b_1 \cdots b_n) = R$.
\end{lemma}
\begin{proof}
Let $d \in R$ be such that $cd = a+b$.  Then
\[ (a,c) \supset (a,cd) = (a,a+b) = (a,b) = R. \qedhere \]
%Similarly $( b, c ) = R$.  b) This is a %standard fact \cite[Lemma 3.17]{Clark-CA}.
\end{proof}

\begin{lemma}
\label{POLLACKLEMMA}
Let $R$ be a domain, not a field, satisfying Condition (E).  For any $at + b \in R[t]$ with $a \in R^{\bullet}$, there is $x \in R$ such that $ax+b$ is a nonzero nonunit.
\end{lemma}
\begin{proof}
Put $P(t) = at+b$.  If $b = 0$, take any nonzero nonunit $x \in R$.  If $b \in R^{\times}$, by Condition (E) there is $x \in R$ such that $b^{-1}ax + 1 \notin R^{\times}$ so $P(x) = b(b^{-1}ax+1)$ is a nonzero nonunit.  If $b \in R$ is a nonzero nonunit, take $x = 0$.
\end{proof}
\noindent
The proof of the following result was suggested to me by Paul Pollack.

\begin{thm}
\label{POLLACKTHM}
Let $R$ be an atomic domain satisfying Condition (E), let $I$ be a nonzero ideal of $R$, and let $H$ be a proper subgroup of $(R/I)^{\times}$.  Then there are infinitely many pairwise comaximal irreducibles $f$ such that the class of $f$ modulo $I$ lies in $(R/I)^{\times} \setminus H$.
\end{thm}
\begin{proof}
Let $r: R \ra R/I$ be the quotient map, let $\alpha \in R$ be such that $r(\alpha) \in (R/I)^{\times} \setminus H$, and let $\beta \in R$ be such that $\alpha \beta -1 \in I \setminus \{0\}$.  Inductively, assume that we have pairwise irreducibles $f_1,\ldots,f_n$ of $R$ such that 
$(f_i,\alpha) = (f_i,I) = R$ for all $i$ and such that $r(f_i) \notin H$.   Let \[P(t) = (\alpha t+1)(\alpha \beta-1)f_1 \cdots f_n + \alpha \in R[t]. \]  (We need to include the base case $n = 0$, and in this case $f_1 \cdots f_n = 1$.)  By Lemma \ref{POLLACKLEMMA} there is $x \in R$ such that
\[ y = (\alpha x+1)(\alpha \beta -1)f_1 \cdots f_n + \alpha  \]
is a nonzero nonunit, so we get an irreducible factorization
\[ y = g_1 \cdots g_s \]
with $s \geq 1$.  Then
\[r(g_1) \cdots r(g_s) = r(y) = r(\alpha) \in (R/I)^{\times} \setminus H, \]
so $(g_j,I) = 1$ for all $j$ and there is at least one $g_j$, say $g_1$, such that $r(g_1) \notin H$.  Now $g_1$ cannot be associate
to any $f_i$; if so $g_1$ and hence also $f_i$ would divide $\alpha$: if $\alpha \in R^{\times}$ this contradicts the irreducibility of $f_i$; if not, this contradicts $(f_i, \alpha) = 1$.  Moreover $y \equiv -f_1 \cdots f_n \pmod{\alpha}$ so $y \in (R/\alpha)^{\times}$, hence also $g_1 \in (R/\alpha)^{\times}$, i.e., $(g_1,\alpha) = R$.  Finally, since
\[ (\alpha x + 1)(\alpha \beta -1)f_1 \cdots f_n \equiv -f_1 \cdots f_n \pmod{\alpha}, \]
we have $((\alpha x + 1)(\alpha \beta-1)f_1 \cdots f_n,\alpha) = R$, so by Lemma \ref{LITTLECRTLEMMA} we have 
\[(g_1,(\alpha x + 1)(\alpha \beta -1)f_1 \cdots f_n) = R \] so $( g_1, f_i ) = R$ for all $i$. Thus we may take $f_{n+1} = g_1$, completing the induction.
\end{proof}

\noindent
When $R = \Z$, we get: for any proper subgroup $H \subsetneq (\Z/N\Z)^{\times}$, there are infinitely many
prime numbers $p$ such that $\pm p \pmod{N} \notin H$.  Moreover, in this classical case one can run the argument
with positive integers only and so get rid of the annoying $\pm$.  This is a special case of Dirichlet's theorem on primes in arithmetic
progressions.  It is an observation of A. Granville -- unpublished by him, but reproduced in \cite[Thm. 1.16]{Pollack} -- that this case can be proved in an elementary ``Euclidean'' way.  The special case of trivial $H$ -- for all $N \geq 3$ there are infinitely many primes $p \not \equiv 1 \pmod{N}$ -- is older and better known.  It is also simpler -- just consider $N p_1 \cdots p_{n-1} - 1$.  This case does not use that $\Z$ is a UFD, but Granville's argument does.  The most auspicious replacement for coprimality arguments is by comaximality, and that is what we've done here.

\section{A ``Topological'' Interlude}

\subsection{Furstenberg's Lemma}
\textbf{} \\ \\ \noindent
In this section we will give several proofs of the following result.

\begin{thm}
\label{BIGFURST}
Let $R$ be a Furstenberg domain with at least one and only finitely many irreducibles $f_1,\ldots,f_n$.  Then: \\
a) We have $\# R^{\times} = \# R$.  \\
b) More precisely there is a nonzero ideal $I$ of $R$ such that $1+I \subset R^{\times}$.
\end{thm}
\noindent
Theorem \ref{BIGFURST} is the contrapositive of part b) of the Euclidean Criterion, without the information on comaximality.  The proofs that we give here are inspired by the famous paper of H. Furstenberg \cite{Furstenberg55}.   The essential core of his argument is the
observation that in $\Z$ the set of elements not divisible by any prime number is $\pm 1$.  Notice that has nothing to do with the natural ordering of $\Z$ that underlies
most of the classical proofs of Euclid's Theorem.  In fact the property of $\Z$ being used is that $\Z$ is a Furstenberg domain.

\begin{lemma}(Furstenberg's Lemma) \textbf{} \\
a) A domain $R$ is a Furstenberg domain iff $R^{\times} = \bigcap_{f \text{ irreducible }} R \setminus (f)$. \\
%
% there is a set $I = \{f_i\}$ of nonassociate %irreducibles such that
%\[ R^{\times} = \bigcap_{i \in I} R \setminus (f_i). %\]
b) In a Furstenberg domain with at least one and only finitely many irreducibles $f_1,\ldots,f_n$, we have
$\bigcap_{i=1}^n (R \setminus (f_i)) = R^{\times}$.
\end{lemma}
\noindent
The proof is virtually immediate and is left to the reader.

\subsection{Following Furstenberg}
\textbf{} \\ \\ \noindent
Let $R$ be a domain.  By Fact \ref{FACT}d), for each $x \in R$, the family \[\mathcal{C}(x) = \{x + I \mid I \text{ is a nonzero ideal of } R\}\] is closed under finite intersections, so $\{\mathcal{C}(x)\}_{x \in X}$ is a system of neighborhood bases for a topology on $R$ -- let us call it the \textbf{adic topology} -- in which $U \subset R$ is open iff for all $x \in U$ there is a nonzero ideal $I$ with $x+I \subset U$.  By Fact \ref{FACT}c), every nonempty open has cardinality $\# R$.
\\ \\
\emph{Proof of Theorem \ref{BIGFURST}}: let $R$ be a Furstenberg domain with at least one and only finitely many irreducibles $f_1,\ldots,f_n$.
Then each $(f_i)$ is open, hence its complement $R \setminus (f_i)$, being a union of cosets of $(f_i)$, is also open.  By Furstenberg's Lemma $R^{\times} = \bigcap_{i=1}^n (R \setminus (f_i))$
is open.  Since $1 \in R^{\times}$, we have
$\# R^{\times} = \# R$.  More precisely, $R^{\times} \supset 1 + I$ for some nonzero ideal of $R$.

%so $\# R^{\times} = \# R$.  More precisely there is a %nonzero ideal $I$ such that $1 + I \subset %R^{\times}$.
 %hence $1 + R\alpha \subset R^{\times}$ for some $\alpha \in R$: Condition (E) does not hold. $\qed$

\subsection{Following Cass-Wildenberg}
\textbf{} \\ \\ \noindent
Let $R$ be a domain, and let $\F_2$ be the field of two elements.  For an ideal $I$ of $R$, a function $f: R \ra \F_2$ is \textbf{I-periodic} if $f(x+y) = f(x)$ 
for all $x \in X$ and $y \in I$.

\begin{lemma}
Let $R$ be a domain, and let $I,I_1,\ldots,I_n$ be nonzero ideals of $R$. \\
a) If $I_2 \subset I_1$ and $f: R \ra \F_2$ is $I_1$-periodic, it is also $I_2$-periodic. \\
b) If for all $1 \leq i \leq n$, $f_i: R \ra \F_2$ is $I_i$-periodic, then the pointwise product $f_1 \cdots f_n: R \ra \F_2$ is
$I_1 \cdots I_n$-periodic.  \\
c) If $f: R \ra \F_2$ is $I$-periodic, then for all $x \in R$, we have
\[ \# \{y \in R \mid f(y) = f(x)\} = \# R. \]
\end{lemma}
\begin{proof}
a) This is immediate from the definition.   \\
b) Certainly $f_1 \cdots f_n$ is $\bigcap_{i=1}^n I_i$-periodic, and $\bigcap_{i=1}^n I_i \supset I_1 \cdots I_n$.  Apply part a). \\
c) Choose a nonzero $\alpha \in I$.  Then $f(x+R\alpha) = f(x)$, and $\# R \alpha = \# R$.
\end{proof}

%Let $R$ be a ring.  For an ideal $I$ of $R$, a function $f: R \ra \F_2$ is \textbf{I-periodic} if for all $x \in R$ and $y \in I$ we have
%$f(x+y) = f(x)$.\footnote{The only property of $\F_2$ that is being used here is that it is a domain.  So we've taken the smallest possible %domain
%to show that nothing interesting is happening from that end.} A function
%$f: R \ra \F_2$ is \textbf{periodic} if it is $I$-periodic for some nonzero ideal $I$.  Thus $f$ is $I$-periodic iff it comes from %precomposing a function $R/I \ra \F_2$ with the natural map $R \ra R/I$.  Now we observe:
%\\ \\
%$\bullet$ If $J \subset I$ and $f$ is $J$-periodic, then it is also $I$-periodic. \\
%$\bullet$ Since every nonzero ideal contains a principal ideal, a function $f: R \ra \F_2$ is periodic iff it is $(\alpha)$-periodic for
%some $\alpha \in R^{\bullet}$.  \\
%$\bullet$ If $R$ is a domain and $f,g: R \ra \F_2$ are periodic, then so is the pointwise product $fg: R \ra \F_2$.  Indeed, if
%$f$ is $I$-periodic and $g$ is $J$-periodic, then $fg$ is $I \cap J$ periodic.  Since $I \cap J \supset IJ \supsetneq (0)$, $fg$ is
%also $IJ$-periodic, hence periodic. \\
%$\bullet$ If $R$ is a domain and $f: R \ra \F_2$ is periodic, then for all $x \in R$ we have \[\# \{y \in R \mid f(y) = f(x)\} = \# R. \]  %Indeed,
%there is $\alpha \in R^{\bullet}$ such that $f(x+R\alpha) = f(x)$, and $\# R \alpha = \# R$.
\noindent
\emph{Proof of Theorem \ref{BIGFURST}}: \\
Step 1: For $1 \leq i \leq n$, let $\chi_i: R \ra \F_2$ be the characteristic function of $(f_i)$; put \[\chi = \prod_{i=1}^n (1-\chi_i). \] Each $\chi_i$ is $(f_i)$-periodic, hence so too is $1-\chi_i$, and thus $\chi$ is $(f_1\cdots f_n)$-periodic.  Moreover $\chi$ is the characteristic function of $\bigcap_{i=1}^n (R \setminus (f_i)) = R^{\times}$.  \\
Step 2: Since $\chi(1) = 1$, $\# R^{\times} = \{x \in R \mid \chi(x) = 1\} = \# R$: part a). \\
Step 3: More precisely $\chi(1+R f_1 \cdots f_n) = 1$, so $R f_1 \cdots f_n + 1 \subset R^{\times}$: part b). $\qed$

\subsection{Following Mercer}
\textbf{} \\ \\ \noindent
Let $R$ be a domain.  Call a subset $X \subset R$ \textbf{lovely} if it is of the form $x+I$ for $x \in R$ and a nonzero ideal $I$ of $R$, i.e., if it is a coset of a nonzero ideal.  Call a subset $X \subset R$ \textbf{pleasant} if it is a union of lovely subsets.   If $I$ is a nonzero ideal of $R$, then $R \setminus I$ is a union of cosets of $I$ hence pleasant.  If $X,Y \subset R$ are pleasant sets and $x \in X \cap Y$, there are nonzero ideals $I,J$ of $R$ such that $x+I \subset X$ and $x+J \subset Y$.  By Fact \ref{FACT}d) $x + (I \cap J) = (x+I) \cap (x+J)$ is a lovely subset of $X \cap Y$ containing $x$.  So $X \cap Y$ is pleasant.  By Fact \ref{FACT}c), every nonempty pleasant subset has cardinality $\# R$.
\\ \\
\emph{Proof of Theorem \ref{BIGFURST}}: let $R$ be a Furstenberg domain with at least one and only finitely many irreducibles $f_1,\ldots,f_n$.  By Furstenberg's Lemma, $R^{\times}$ is the finite intersection of complements of nonzero ideals so is pleasant.  Since $1 \in R^{\times}$, we have
$\# R^{\times} = \# R$.  More precisely, $R^{\times} \supset 1 + I$ for some nonzero ideal of $R$.

\subsection{Debriefing}
\textbf{} \\ \\ \noindent
The three proofs given above are generalizations of the proofs of Euclid's Theorem given by Furstenberg \cite{Furstenberg55}, Cass-Wildenberg \cite{Cass-Wildenberg03} and Mercer \cite{Mercer09}.   The latter two works take the detopologization of Furstenberg's proof as their goal.
%Mercer was not at the time aware of %\cite{Cass-Wildenberg03}, but in an Editor's Endnote %in the May 2010 issue of the \emph{Monthly} he %acknowledges the connection.  One might call Mercer's %proof a ``deperiodification'' of Cass-Wildenberg's.  %Alternately, one can see both the proofs of 5Furstenberg and
%Cass-Wildenberg as a ``recortication'' of Mercer's %proof.
\\ \indent
Our presentation of the argument of $\S$3.4 differs superficially from Mercer's.  We chose the words ``lovely'' and ``pleasant'' precisely because they do not have a commonly understood technical mathematical meaning: had we said ``basic'' and ``open'' then the reader's attention would have been drawn to the fact that since the basic sets are closed under finite intersections, they form the base of a topology.  Mercer's exposition takes pains to point out that the underlying fact here is just that finite intersections of unions are
unions of finite intersections.  Of course this is a basic logical principle: conjunctions distribute over disjunctions and conversely.  Like many basic logical principles it is completely innocuous when used in context (as in our version of the argument).  That the pleasant sets form a topology on $R$ is no more and no less than a crisp enunciation of the facts we need to check in the first part of the proof.
%\footnote{To be precise it is \emph{slightly} more: the proof does not require us to check that every union of pleasant sets is pleasant -- but %that is completely clear from the definition!} 
I find it quite striking (and pleasant!) that the facts can be enunciated in this way, but I must now agree with those who have claimed that there is no \emph{essential topological content} in Furstenberg's argument.\footnote{Furstenberg does not claim a \emph{topological proof of the infinitude of the primes} but rather a \emph{``topological'' proof of the infinitude of the primes}.}
 \\ \indent
 The use of periodic functions involves slightly more packaging, but of a standard kind: it is well known that the Boolean ring $2^R$ of subsets of $R$ can be represented as the ring $\operatorname{Maps}(R,\F_2)$
 with pointwise addition and multiplication.  We recommend wikipedia and Glaymann \cite{Glaymann67} as references.  Glaymann develops this correspondence and applies it to prove such identities as  $A \Delta B = C \iff B \Delta C = A \iff C \Delta A = B$...in a manner intended to be used in the high school classroom.  This is an interesting snapshot of ``the new math'' near its zenith.

\subsection{The Ubiquitous Theorem}
\textbf{} \\ \\
Here is a result that complements Theorem \ref{BIGFURST}.  It is not deep, but it will play a recurring role for us as a common intersection of various constructions and themes.   The first proof 
that we give follows the ``topological conceit'' of this section.  We will give other, simpler, proofs later on.

%\textbf{} \\ \\ \noindent
%In this section we will give a proof of the following %result.

\begin{thm}
\label{GOLOMB}
Let $R$ be a domain, not a field, with only finitely many maximal ideals $\mm_1,\ldots,\mm_n$.  Then: \\
a) We have $\# R^{\times} = \# R$. \\
b) More precisely there is a nonzero ideal $I$ of $R$ such that $1 + I \subset R^{\times}$.
% we have $1+ \bigcap_{i=1}^n \mm_i \subset %R^{\times}$.
\end{thm}
\begin{proof}
We endow $R$ with the topology for which, for $x \in R$, $\mathcal{C}(x) = \{x + \mm \mid \mm \text{ is a maximal ideal of } R \}$ is a neighborhood subbase at $x$: that is, $U \subset R$ is open iff for all $x \in U$ there is a subset $J \subset \{1,\ldots,n\}$ such that
\[ \bigcap_{i \in J} (x + \mm_i) = x + \bigcap_{i \in J} \mm_i \subset U. \]
Fact \ref{FACT} gives $\bigcap_{i \in J} \mm_i \supsetneq (0)$, so every nonempty open has cardinality $\# R$.  Each $R \setminus \mm_i$, being a union of cosets of $\mm_i$, is also open.  Therefore \[R^{\times} = \bigcap_{i=1}^n (R \setminus \mm_i)\] is open.  Since $1 \in R^{\times}$ we have $\# R^{\times} = \# R$.  More precisely there is a subset $J \subset \{1,\ldots,n\}$ such that $1+\bigcap_{i \in J} \mm_i \subset R^{\times}$, and thus also $1 + \bigcap_{i=1}^n \mm_i \subset R^{\times}$.
\end{proof}

\subsection{Supplement: Further Topologies on a Domain}
\textbf{} \\ \\ \noindent
Here is a common generalization of Theorems \ref{BIGFURST} and \ref{GOLOMB}: let $\mathcal{J}$ be a family of nonzero ideals of a domain $R$, and suppose there are $I_1,\ldots,I_n \in \mathcal{J}$ such that $R^{\times} = \bigcap_{i=1}^n (R \setminus I_i)$.  Then $1 + \bigcap_{i=1}^n I_i \subset R^{\times}$, so in particular $\# R^{\times} = \# R$.
\\ \indent
Look again at Theorem \ref{BIGFURST}: instead of taking $\mathcal{J}$ to be the family of all nonzero ideals, we could take $\mathcal{J} = \{ (f_1),\ldots,(f_n)\}$ and endow $R$ with the unique translation-invariant topology with $\mathcal{J}$ as a neighborhood subbase at $0$.  This coarsens the adic topology\footnote{The adic topology on a domain is always Hausdorff, but in a Furstenberg domain with finitely many irreducibles, this new topology is not.} so that being open yields the sharper conclusion $1 + \bigcap_{i=1}^n (f_i) \subset R^{\times}$.  In particular
$1 + ( f_1 \cdots f_n ) \subset R^{\times}$.  We are back to a version of Euclid's argument.
\\ \indent
The adic topology on $\Z$ is not very interesting \emph{as a topological space}: it is countably infinite, metrizable, totally disconnected and without isolated points, hence homeomorphic to the Euclidean topology on $\Q$.  In \cite{Golomb59}, Golomb proved Euclid's Theorem using the topology on $\Z^+$ with base the one-sided arithmetic progressions $\{an + b \mid n \in \Z^+\}$ for \emph{coprime} $a,b \in \Z^+$. Golomb's topology makes $\Z^+$ into a countably infinite connected Hausdorff space...which is already interesting. 
%(If you don't see why, try to construct any such topological space.)
\\ \indent
In a domain $R$ that is not a field, we may consider the \textbf{Golomb topology} with neighborhood base at $x \in R$ given by \[\mathcal{C}(x) = \{x + I \mid I \text { is a nonzero ideal with }( x, I ) = R \}. \]
%By \cite[Lemma 3.17]{Clark-CA}, if $( x,I %) = ( x,J ) = R$ then $( x,IJ %) = R$, so $\mathcal{C}(x)$ is closed under %finite intersections.
In this topology every maximal ideal is closed, so in a domain that is not a field with only finitely many maximal ideals $\mm_1,\ldots,\mm_n$, $R^{\times}$ is open and thus contains $1+I$ for some nonzero ideal $I$.  We get another proof of Theorem \ref{GOLOMB}.
\\ \indent
The Golomb topology is never Hausdorff: in fact $\overline{ \{0\} } = R$.  However, the induced topology on $R^{\bullet}$ can be (it is for $\Z$).  We leave further exploration to the reader.

 %Since $\overline{\{0\}} = R$ this topology is not %Hausdorff, so we may want to look at the subspace %topology on $R^{\bulet}$

\section{Connections With Ideal Theory}
\noindent
For a ring $R$, we denote by $\MaxSpec R$ the set of all maximal ideals of $R$.

\subsection{Comaximal Ideals}

\begin{lemma}
Let $\{I_n\}_{n=1}^{\infty}$ be a sequence of pairwise comaximal proper ideals in a ring $R$.  Then $\MaxSpec R$ is infinite.
\end{lemma}
\begin{proof}
For $n \in \Z^+$, let $\mm_n$ be a maximal ideal containing $I_n$.  If for $n_1 \neq n_2$ we had
$\mm_{n_1} = \mm_{n_2}$ then $R = I_{n_1} + I_{n_2} \subset \mm_{n_1}$,
contradiction.
\end{proof}
\noindent
In particular, part a) of the Euclidean Criterion implies that a domain that is not a field and that satisfies Condition (E) has infinitely many maximal ideals.  Thus we get another proof of Theorem \ref{GOLOMB}....but by no means our last.

\subsection{Euclid Meets Jacobson}
\textbf{} \\ \\ \noindent
Now is the time to examine the more explicitly ideal-theoretic statement of Condition (E): for all nonzero ideals $I$, we have $1 + I \not \subset R$.  Some readers will now see -- or will have already seen -- the connection with the Jacobson radical, but we will not assume a prior familiarity.  In fact we will use the Euclidean Criterion to motivate a self-contained discussion of this and other ideal-theoretic concepts.

% The rest of the section is devoted to making this connection evident to all.

\begin{prop} \cite[Prop. 4.14]{Clark-CA}
\label{3.2}
For a ring $R$, let
\[ J(R) = \bigcap_{\mm \in \MaxSpec R} \mm, \]
the \textbf{Jacobson radical} of $R$.  For $x \in R$, the following are equivalent: \\
(i) $x \in J(R)$. \\
(ii) For all $y \in R$, $yx+1 \in R^{\times}$.
\end{prop}
\begin{proof}
(i) $\implies$ (ii): By contraposition: suppose there is $y \in R$ such that $z = yx+1 \notin R^{\times}$.  Then $z$ lies in some maximal ideal $\mm$.  If also $x \in \mm$, then $yx \in \mm$ and thus also $z-yx = 1 \in \mm$, contradiction.  So $x$ does not lie in $\mm$ and thus $x \notin J(R)$.  \\
(ii) $\implies$ (i): Again by contraposition: suppose that there is a maximal ideal $\mm$ such that $x \notin \mm$.  Then $\mm \subsetneq ( \mm,x )$, so $( \mm,x  ) = R$.  It follows that there is $m \in \mm$ and $y \in R$ such that $m+yx = 1$.  Thus $(-y)x + 1 = -m \in \mm$ so is not a unit.
\end{proof}
\noindent
We get immediately:

\begin{cor}
\label{3.33}
A ring $R$ satisfies Condition (E) iff $J(R) = (0)$.
\end{cor}
\noindent
This gives a third proof of Theorem \ref{GOLOMB}: if $R$ has only finitely many maximal ideals $\mm_1,\ldots,\mm_n$, then
\[ J(R) = \bigcap_{i=1}^n \mm_i \supset \prod_{i=1}^n \mm_i \supsetneq \{0\}. \]
Apply Corollary \ref{3.33}.
\\ \\
A ring with zero Jacobson radical is called \textbf{semiprimitive}.\footnote{Or \textbf{Jacobson semisimple} or \textbf{J-semisimple}.}
%Well, the name is not so important but that there is %a name is very encouraging: indeed this is a well %studied concept, and we can draw on known theory to %supply further applications of the Euclidean %Criterion.

%In the next sections we consider some classes of %domains with respect to semiprimitivity, getting in %%particular some cases of applicability of the %Euclidean Criterion.  We wish in particular to %address the following:

\subsection{Some Questions and Some Answers}
\textbf{} \\ \\ \noindent
We now raise some natural questions...and answer them.

\begin{ques}
\label{EXTRAQUES}
In part b) of the Euclidean Criterion, must we assume that $R$ is a Furstenberg domain?
\end{ques}

\begin{ques}
\label{QUES2.4}
A semiprimitive domain, not a field, has infinitely many maximal ideals.  Must a domain with infinitely many maximal ideals be semiprimitive?  
%If not always, then what are some useful conditions for this?
\end{ques}

\begin{ques}
\label{QUES2.5}
Let $R$ be a Furstenberg domain.  \\
a) If $R$ is not semiprimitive, can it still have infinitely many atoms? \\
b) Can $R$ have finitely many maximal ideals and infinitely many atoms?
\end{ques}

\begin{example}
\label{EX4.7}
The ring $\overline{\Z}$ of all algebraic integers is \emph{not} a Furstenberg domain.  In fact it is an \textbf{antimatter domain}: there are no irreducibles whatsoever: if $z$ is an algebraic integer then so is $z^{1/2}$, so we can always factor $z = z^{1/2} z^{1/2}$.  Moreover $\overline{\Z}$ is not a field: for all integers $n \geq 2$, if $n \in \overline{\Z}^{\times}$ then $\frac{1}{n} \in \overline{\Z} \cap \Q = \Z$, contradiction.  \\ \indent If $I$ is a nonzero ideal of $\overline{\Z}$ then the constant coefficient of the minimal polynomial of a nonzero element $\alpha \in I$ is a nonzero integer in $I$.  It follows that if $J(\overline{\Z}) \neq 0$ then there is $N \in \Z^+$ that is contained in every $\mm \in \MaxSpec \overline{\Z}$.  Choose a prime number $p \nmid N$.
  Then $p$ is not  a unit in $\overline{\Z}$ -- otherwise $\frac{1}{p} \in \overline{\Z} \cap \Q = \Z$ -- so there is at least one maximal ideal $\mm_p$ of $\overline{\Z}$ containing $p$.  (In fact the set of maximal ideals of $\overline{\Z}$ containing $p$ has continuum cardinality.)  Then
  $\mm_p \supset ( N,p ) = \overline{\Z}$: contradiction.
\end{example}
  \noindent
  So the answer to Question \ref{EXTRAQUES} is \textbf{yes}: a semiprimitive domain that is not a field can have no irreducibles whatsoever.
\\ \\
The following result answers Questions \ref{QUES2.4} and \ref{QUES2.5} for Dedekind domains and shows that the Euclidean Criterion is, in principle, completely efficacious in determining whether a Dedekind domain has infinitely many atoms.

\begin{thm}
\label{3.4}  \label{EUCLIDJACOBSON}\textbf{} \\
For a Dedekind domain $R$ that is not a field, the following are equivalent: \\
(i) $R$ is semiprimitive.  \\
(ii) $R$ has infinitely many maximal ideals.  \\
(iii) $R$ has infinitely many atoms.
\end{thm}
\begin{proof}
We know (i) $\implies$ (ii) in any domain. \\
(ii) $\implies$ (i): in a Dedekind domain, any nonzero element is contained in only finitely many maximal ideals.  So in fact for any infinite subset $M \subset \MaxSpec R$ we have $\bigcap_{\mm \in M} \mm = (0)$.  \\
(i) $\implies$ (iii): Dedekind domains are Noetherian, hence Furstenberg domains, so the Euclidean Criterion applies. \\
(iii) $\implies$ (i): By contraposition: a Dedekind domain with finitely many maximal ideals is a PID \cite[Thm. 20.6]{Clark-CA}, and in a PID there is no distinction between maximal ideals, principal ideals generated by prime elements, and atoms.
\end{proof}

\begin{ques}
Let $K$ be a number field, with ring of integers $\Z_K$.  The set of prime numbers is an infinite sequence of pairwise comaximal nonunits of $\Z_K$, so (as is well known!) $\Z_K$ has infinitely many prime ideals and thus is semiprimitive.  When $K = \Q$ or is imaginary quadratic, the finiteness of $\Z_K^{\times}$ leads to a \emph{direct verification} of Condition (E).  Is there a similarly direct verification for all $K$?
\end{ques}
\noindent
This is a question we will leave to the reader to address.

\begin{prop}
Let $R$ be a Noetherian domain of dimension at most one (nonzero prime ideals are maximal).  If $\MaxSpec R$ is infinite, then $R$ is semiprimitive and thus has infinitely many pairwise comaximal irreducibles.
\end{prop}
\begin{proof}
If $R$ is not semiprimitive, then every maximal ideal $\mm$ of $R$ is a minimal prime ideal of $R/J(R)$.  Since $R$ is Noetherian, so is $R/J(R)$, and a Noetherian ring has only finitely many minimal prime ideals \cite[Thm. 10.13]{Clark-CA}.
\end{proof}
\noindent
A \textbf{Jacobson ring} is a ring in which every prime ideal is the intersection of the maximal ideals containing it.  Since in a domain $(0)$ is prime, a Jacobson domain must be semiprimitive.  Any quotient of a Jacobson ring is again a Jacobson ring.  If $R$ is a Jacobson ring and $S$ is a commutative, finitely generated $R$-algebra then $S$ is a Jacobson ring \cite[Thm. 12.15, 12.21]{Clark-CA}.  So:

\begin{thm}
\label{3.9}
a) A Jacobson Furstenberg domain that is not a field has infinitely many pairwise comaximal irreducibles. \\
b) Let $F$ be a field, and let $\pp$ be a prime but not maximal ideal of $F[t_1,\ldots,t_n]$.  Then the ring $R = F[t_1,\ldots,t_n]/\pp$ -- i.e., a coordinate ring of an integral affine variety of positive dimension -- has infinitely many pairwise comaximal irreducibles.  \\
c) A domain $R$ that is finitely generated over $\Z$ and not a field has infinitely many pairwise comaximal irreducibles.
\end{thm}
\noindent
To sum up: if we want to see a domain that has infinitely many maximal ideals but is not semiprimitive, it cannot be finitely generated over a field, and if Noetherian it must have a nonzero prime ideal that is not maximal.  This cues us up for the following example, 
which gives a negative answer to Question \ref{QUES2.4}.

\begin{example}
Consider the ring $\Z[[t]]$ of formal power series with integral coefficients.  It is not hard to show that $\Z[[t]]$ is an atomic domain.  In fact $\Z[[t]]$ is a Noetherian UFD  
%\cite[p. 48, Thm. 72]{Kaplansky70}
 \cite[Thm. 15.32]{Clark-CA}.  Since $1+(t) \subset \Z[[t]]^{\times}$, the Jacobson radical $J(\Z[[t]])$ contains $(t)$ and is thus nonzero.  
%(For any ring $D$, $J(D[[t]]) = ( J(D),t )$ \cite[Exc. 5.6]{Lam}.)  
Since $J(\Z[[t]])) \neq (0)$, the hypotheses of the Euclidean Criterion \emph{do not} apply.  Nevertheless
there are infinitely many pairwise comaximal prime elements, namely the prime numbers!  Hence there are infinitely many maximal
ideals. 
% Explicitly, we can take $( p, t )$ as $p$ ranges over prime numbers.
\\ \indent
Here we could have replaced $\Z$ with any PID with infinitely many maximal ideals.
\end{example}
\noindent
Thus the answer to Question \ref{QUES2.5}a) is \textbf{yes}: moreover a nonsemiprimitive domain can have infinitely many \emph{comaximal} irreducibles.

\begin{example}
Let $k$ be a field.  Recall that $k[x,y]$ is a UFD, and let $K = k(x,y)$ be its fraction field.  Let $R$ be the subring of $k(x,y)$ consisting of rational functions $\frac{f(x,y)}{g(x,y)}$ that, when written in lowest terms, have $g(0,0) \neq 0$.  Then $R$ is itself a UFD -- factorization in $R$ proceeds as in $k[x,y]$ except that the prime elements $p(x,y) \in k[x,y]$ such that $p(0,0) \neq 0$ become units in $R$ -- in which an element is a unit iff $f(0,0) \neq 0$.  Thus $\mm = \{ \frac{f(x,y)}{g(x,y)} \mid f(0,0) = 0\}$ is the unique
maximal ideal, so $J(R) = \mm$ and $R$ is very far from being semiprimitive.  Nevertheless it has infinitely many prime elements,
e.g. $\{y-x^n\}_{n=1}^{\infty}$.  In more geometric language, the irreducibles are the irreducible curves in the affine plane passing through $(0,0)$.
\end{example}
\noindent
Thus the answer to Question \ref{QUES2.5}b) is \textbf{yes}.  However, there is more to say.  The preceding example can be vastly generalized using the following striking result.
% which deserves to be much better known.

% I found it in a literature search done while preparing this document.

\begin{thm}(Cohen-Kaplansky \cite{Cohen-Kaplansky46})
\label{3.14} \textbf{} \\
Let $R$ be an atomic domain with finitely many atoms.  Then: \\
a) $R$ has only finitely many prime ideals. \\
b) $R$ is Noetherian. \\
c) Every nonzero prime ideal of $R$ is maximal.
\end{thm}
\begin{proof}
a)  In an atomic domain $R$, whenever a prime ideal $\pp$ of $R$ contains a nonzero element $x$, we may factor $x = f_1 \cdots f_r$ into irreducibles and thus see that $\pp$ contains some irreducible element $f$ dividing $x$.  Thus, given any set of generators of a prime ideal $\pp$ we can replace it with a set of irreducible generators.  In a set of generators of an ideal, replacing each element by any one of its associates does not change the ideal generated, and thus if we have only finitely many nonassociate irreducibles we can only generate finitely many prime ideals. \\
b) It follows from the proof of part a) that every prime ideal of $R$ is finitely generated.  By a famous result of Cohen \cite[Thm. 4.26]{Clark-CA}, all ideals are finitely generated.  (This is an instance of the \textbf{prime ideal principle} of Lam-Reyes \cite{Lam-Reyes08}.) \\
c) If not, there are prime ideals $(0) \subsetneq \pp_1 \subsetneq \pp_2$.  Since $R$ is Noetherian, this implies there are infinitely many
prime ideals between $(0)$ and $\pp_2$ \cite[Cor. 8.46]{Clark-CA}.
\end{proof}
\noindent
A \textbf{Cohen-Kaplansky domain} is an atomic domain with finitely many atoms.  The work \cite{Cohen-Kaplansky46} does not give a complete classification: we are left with the case of a Noetherian domain $R$ with finitely many nonzero prime ideals, all of which are maximal.  If $R$ is a Dedekind domain, then by Theorem \ref{3.4} there are only finitely many atoms.  So the remaining case is when $R$ is not integrally closed in its fraction field, in which case the integral closure $\overline{R}$ is a Dedekind domain with finitely many prime ideals \cite[Cor. 18.8]{Clark-CA}.  One might expect that this forces $R$ to be Cohen-Kaplansky.  This need not be the case!

\begin{example}
\label{4.15}
Let $k$ be a field, and consider the subring
\[ R = k[[t^2,t^3]] = k + t^2 k[[t]] \]
of the formal power series ring $k[[t]]$.  
For  $0 \neq f = \sum_{n=0}^{\infty} a_n t^n \in k[[t]]$, we define $v(f)$ to be the least $n$ such that $a_n \neq 0$.  Then $v$ is a discrete valuation on $k[[t]]$, and the only nonzero prime ideal of $k[[t]]$ is
$(t) = \{f \in R \mid v(f) > 0\} \cup \{0\}$.  In particular, $k[[t]]$ is a PID.  So is the (isomorphic!) subring $k[[t^2]]$, 
and $\{1,t^3\}$ is a generating set for $R$ as a $k[[t^2]]$-module, so by standard PID structure theory, every ideal of $R$ can be generated by two elements.  Thus $R$ is Noetherian, hence atomic.
For $f = a_0 + \sum_{n=2}^{\infty} a_n t^n \in R$, we have $f \in R^{\times} \iff a_0 \neq 0$, and thus 
\[\mm = \{\sum_{n=2}^{\infty} a_n t^n \}  = ( t^2,t^3 )\]
is the unique maximal ideal of $R$.   We will give a complete description of the atoms of $R$.  First we claim that $f \in R$ is irreducible iff $v(f) \in \{2,3\}$.  Indeed a nontrivial factorization $f = xy$ involves
$v(x),v(y) \geq 2$ hence $v(f) \geq 4$; conversely, if $v(f) \geq 4$ then $f = t^2 \frac{f}{t^2}$ is a nontrivial factorization.  Since $k^{\times} \subset R^{\times}$, every irreducible is associate to one of the form

\[ t^2 + \sum_{n \geq 3} a_n t^n, \ (v(f) = 2 \text{ case}) \]
or one of the form
\[  t^3 + \sum_{n \geq 4} a_n t^n,  \ (v(f) = 3 \text{ case}).  \]
Associate elements have the same valuation, so certainly no irreducible of the first type is associate to an irreducible of the second type. We claim that $t^2 + \sum_{n \geq 3} a_n t^n$ is associate to $t^2 + \sum_{n \geq 3} b_n t^n$ iff $a_3 = b_3$ and $t^3 + \sum_{n \geq 3} a_n t^n$ is associate to $t^3 + \sum_{n \geq 3} b_n t^n$ iff $a_4 = b_4$.  This can be done by direct computation:
\[ (t^2 + a_3 t^3 + a_14 t^4 + a_5 t^5 + \ldots )(1 + u_2 t^2 + u_3 t^3 + \ldots) \]
 \[= t^2 + a_3 t^3 + (a_4 + u_2) t^4 + (a_5 + a_3 u_2 + u_3) t^5 + \ldots, \]
so $a_3 = b_3$ and there is a unique choice of $u_2,u_3,\ldots$ leading to $a_n = b_n$ for all $n \geq 4$.  The $v(f) = 3$ case is similar.  Thus there are precisely $2 \# k$ atoms, and $R$ is Cohen-Kaplansky iff $k$ is finite.
\end{example}

\begin{example}
(Anderson-Mott \cite[Cor. 7.2]{Anderson-Mott92}) For a prime power $q$ and $d,e \in \Z^+$, the ring $R = \F_q + t^e \F_{q^d}[[t]]$ is a Cohen-Kaplansky domain with exactly one nonzero prime ideal and exactly $e \frac{q^d-1}{q-1} q^{d(e-1)}$ irreducibles, none of which are prime unless $(d,e) = (1,1)$.
\end{example}
\noindent
The paper \cite{Cohen-Kaplansky46} was mostly forgotten for many years, until the breakthrough work of Anderson and Mott \cite{Anderson-Mott92} gave a complete characterization of Cohen-Kaplansky domains.  In fact they give $14$ characterizations!  Here is one:

\begin{thm}(Anderson-Mott \cite{Anderson-Mott92})
\\ For an atomic domain $R$, the following are equivalent: \\
(i) $R$ is a Cohen-Kaplansky domain. \\
(ii) $R$ is Noetherian of dimension at most one (nonzero prime ideals are maximal), has finitely many prime ideals, the integral closure $\overline{R}$ of $R$ is finitely generated as an $R$-module, $\# \MaxSpec \overline{R} = \# \MaxSpec R$, and for all nonprincipal ideals $\mm \in \MaxSpec R$,
$R/\mm$ is finite.
\end{thm}

\begin{example}
Let $k$ be a field of characteristic different from $2$ or $3$, and consider: \\
$\bullet$ $R_1$: the localization of $k[x,y]/(y^2 - x^3 - x)$ at $\mm_0 = ( x,y )$. \\
$\bullet$ $R_2$: the localization of $k[x,y]/(y^2-x^3-x^2)$ at $\mm_0 = ( x,y )$.  \\
$\bullet$ $R_3$: the localization of $k[x,y]/(y^2-x^3)$ at $\mm_0 = ( x,y )$. \\
Then:
\\
$\bullet$ $R_1$ is always Cohen-Kaplansky (it is a Dedekind domain with one maximal ideal). \\
$\bullet$ $R_2$ is never Cohen-Kaplansky ($ \# \MaxSpec \overline{R_2} = 2 > 1 = \# \MaxSpec R_2$). \\
$\bullet$ $R_3$ is Cohen-Kaplansky iff $k$ is finite.
\end{example}

\subsection{Euclid Beyond Atomicity}
\textbf{} \\ \\ \noindent
In the case of an atomic domain, the part of the Euclidean Criterion that yields infinitely many maximal ideals is much weaker than the Cohen-Kaplansky Theorem.  However, there is life beyond atomic domains.

\begin{example}
\label{EXAMPLE4.19}
Let $\Hol(\C)$ be the ring of entire functions $f: \C \ra \C$.  For $f \in \Hol(\C)$, put $Z(f) = \{z \in \C \mid f(z) = 0\}$.  
 If $f,g \in \Hol(\C)^{\bullet}$, then $Z(f)$ and $Z(g)$ are countable sets, hence so is $Z(fg) = Z(f) \cup Z(g)$, so $fg \neq 0$.  Thus $H(\C)$ 
is a domain.  The map $z_0 \in \C \mapsto ( z - z_0 )$ gives a bijection from $\C$ to the atoms of $\Hol(\C)$.  An element $f \in \Hol(\C)$ 
is a unit iff $Z(f) = \varnothing$, and a nonzero nonunit $f$ is a (finite!) product of atoms iff $Z(f)$ is finite and nonempty.
% (Amusingly, 
%the Weierstrass Factorization Theorem reaches the opposite conclusion: every $f \in \Hol(C)^{\bullet}$ is a product of linear factors 
%$(z-z_0)$ and a nowhere vanishing holomorphic function -- but if $f$ has infinitely many zeros, one gets a convergent 
%\emph{infinite} product.)  
\\ \indent So $\Hol(\C)$ is not atomic -- consider e.g. $f(z) = \sin z$ -- but it is Furstenberg: if $f$ is a nonzero nonunit, then $f$ vanishes at some $z_0 \in \C$ and thus is divisible by the irreducible element $z-z_0$.  Moreover $\Hol(\C)$ satisfies Condition (E): 
if $f \in \Hol(\C)^{\bullet}$ then there is $w \in \C$ such that $f(w) \neq 0$.  Let $g = z - w - \frac{1}{f(w)}$.  Then $(gf + 1)(w) = 0$, 
so $gf+1 \notin \Hol(\C)^{\times}$.   Thus the Euclidean Criterion applies in $\Hol(\C)$.
\end{example}

% as we see in the following result.

\begin{thm}
\label{LASTTHM}
Let $1 \leq \alpha \leq \beta \leq \gamma$ be cardinal numbers.  There is a domain $R$ satisfying all of the following properties: \\
(i) $R$ is a B\'ezout domain: every finitely generated ideal is principal.  \\
%(Thus every irreducible in $R$ is prime.) \\
(ii) $R$ has exactly $\alpha$ atoms, each of which is a maximal ideal. \\
(iii) $R$ has exactly $\beta$ maximal ideals. \\
(iv) $R$ has exactly $\gamma$ nonzero prime ideals.  \\
(v) $R$ is an atomic domain iff $\alpha = \beta = \gamma < \aleph_0$. \\
(vi) $R$ is a Furstenberg domain iff $\alpha = \beta$. \\
(vii) $R$ is semiprimitive iff $\beta \geq \aleph_0$.
\end{thm}
\noindent
We postpone the proof of Theorem \ref{LASTTHM} in order to discuss its significance.  By taking $\alpha = \beta$  and $\gamma \geq \aleph_0$ we get Furstenberg domains with any number $\alpha \geq 1$ of irreducibles and any number $\gamma \geq \max(\alpha,\aleph_0)$ nonzero prime ideals.  In particular, a Furstenberg domain can have any finite, positive number of irreducibles and any infinite number of prime ideals, so the Cohen-Kaplansky Theorem does not extend from atomic domains to Furstenberg domains.  For any $\alpha = \beta \geq \aleph_0$ and $\gamma \geq \alpha$ we get a semiprimitive Furstenberg domain that is not an atomic domain.
\\ \\
Now we come to the proof of Theorem \ref{LASTTHM}, which requires somewhat more specialized results.  A completely self-contained presentation would require more space than we want to devote here.  So we will make use of the material of  \cite[Ch. II and III]{Fuchs-Salce}, and our treatment will be at the level of a detailed sketch.  \\
\indent Let $R$ be a domain with fraction field $K$.  To $x \in K^{\bullet}$ we attach the \textbf{principal fractional ideal} 
$(x) = \{ax \mid a \in R\}$.  When $x \in R$, this coincides with the usual notion of a principal ideal.  For $x,y \in K^{\bullet}$ we have $(x) = (y)$ iff there is $u \in R^{\times}$ 
such that $y = ux$.  The principal fractional ideals of $K$ form a commutative group under pointwise multiplication: we have 
$(x)(y) = (xy)$.  We call this the \textbf{group of divisibility} of $R$ and denote it $G(R)$.  It is partially ordered by reverse 
inclusion: that is, for $x,y \in K^{\bullet}$ we put $(x) \leq (y)$ iff $(y) \supset (x)$.  This order reversal is actually rather familiar: 
for $x,y \in K^{\times}$, we write $x \mid y \iff \frac{y}{x} \in R$, and then we have $x \mid y$ if  $(x) \supset (y)$: \emph{to contain is 
to divide}. \\ \indent
 Let $\{G_i\}_{i \in I}$ be an indexed family of nonzero totally ordered commutative groups, and let $G = \bigoplus_{i \in I} G_i$ be the direct sum endowed with the pointwise partial ordering: $x \leq y$ iff
$x_i \leq y_i$ for all $i \in I$.  Let $\pi_i: G \ra G_i$ be projection onto the $i$th coordinate.  By the Kaplansky-Jaffard-Ohm Theorem \cite[Thm. III.5.3]{Fuchs-Salce} there is a B\'ezout domain $R$ and an isomorphism $\varphi: G(R) \stackrel{\sim}{\ra} G$ of partially ordered commutative groups.  See \cite[Example III.5.4]{Fuchs-Salce}.  Let $v$ be the composite $K^{\times} \ra K^{\times}/R^{\times} \stackrel{\varphi}{\ra} \bigoplus_{I \in I} G_i$.  Then the maximal ideals of $R$ are precisely $\mm_i = \{ x \in R \mid (\pi_i \circ v)(x) > 0\} \cup \{0\}$ for $i \in I$.  Thus no element of $R^{\bullet}$ lies in infinitely many maximal ideals, so $R$ is semiprimitive iff $I$ is infinite.  \\ \indent
An \textbf{atom} in a partially ordered commutative group is a minimal positive element.  This is a direct generalization of our 
previous use of the term: if $R$ is a domain, the minimal positive elements of the group of divisibility $G(R)$ are precisely 
the principal fractional ideals $(x)$ for an irreducible element $x \in R$.  For every atom $x \in G$, there is $i \in I$ such that $x_i$ is an atom of $G_i$ and $x_j = 0$ for all $j \neq i$, and conversely all such elements give atoms of $G$.  Since each $G_i$ is totally ordered, it has at most one atom, the \emph{least} positive element of $G_i$ if such an element exists.  It follows that $R$ is Furstenberg iff each $G_i$ has a least positive element.  Similarly, a nonzero nonunit $x \in R$ factors into irreducibles iff
$v(x) \in G$ is a sum of atoms iff for all $i \in I$, $G_i$ has a least positive element $a_i$ and $v_i(r) = n a_i$ for some $n \in \Z^+$.  Thus $R$ is an atomic domain iff each $G_i \cong \Z$.  \\ \indent
The domain $R$ is \textbf{h-local}: each nonzero prime ideal is contained in a unique maximal ideal \cite[\emph{loc. cit.}]{Fuchs-Salce}.  The nonzero prime ideals contained in $\mm_i$ correspond bijectively to the proper convex subgroups of $G_i$.  (A subset $Y$ of a totally ordered set $X$ is convex if
for all $x < y < z \in X$, if $x,z \in Y$ then also $y \in Y$.)  We will take each $G_i$ to be a lexicographic product of copies of subgroups of $(\R,+)$ indexed by an ordinal $\eta$.  Then the convex subgroups of of $G_i$ are precisely $\{H_{\delta}\}_{0 \leq \delta \leq \eta}$, where $H_{\delta}$ is the set of all elements of $G_i$ with $j$-coordinate zero for all $j < \delta$.  So there are $\# \eta$ nonzero prime ideals in $\mm_i$.
\\ \indent
We will take a family of nonzero totally ordered commutative groups $G_i$ parameterized by $i \in \beta$: this gives us $\beta$ maximal ideals, and $R$ is semiprimitive iff $\beta \geq \aleph_0$.  We are left to choose the groups $G_i$ in terms of $\alpha$ and $\gamma$ so as to attain the other assertions.  We define an ordinal $\eta$: if $\gamma$ is finite, it is the positive integer
$\gamma-\beta + 1$; if $\gamma$ is infinite, it is the successor ordinal to $\gamma$ (what matters in this case is that $\eta$ is a well-ordered set of cardinality $\gamma$ and with a largest element).  There are cases: \\
$\bullet$ If $\alpha = \beta = \gamma < \aleph_0$, we take $R$ to be a PID with $\gamma$ nonzero prime ideals.  \\
$\bullet$ If $\alpha = \beta$ and $\gamma \geq \min(\beta+1,\aleph_0)$ we take $G_i = \Z$ for all $0 < i \in \beta$.  We take $G_0$ to be the Cartesian product of copies of $\Z$ indexed by $\eta$, endowed with the lexicographic ordering.  Then $G_0$ has a least positive element: the element that is $0$ in all factors but the last and $1$ in the last factor.  So all $G_i$ have least elements and $R$ is a Furstenberg domain.  Moreover $\eta \geq 2$ so $G_0 \not \cong \Z$ and $R$ is not an atomic domain.  It has $(\beta-1) + \# \eta = \gamma$ nonzero prime ideals.  \\
$\bullet$ If $\alpha < \beta$, we take $G_0$ to be the Cartesian product of copies of $\Z$ indexed by $\eta$, for $1 \leq i < \alpha$ we take $G_i = \Z$, and for $i \geq \alpha$ we take $G_i = \R$.
%This completes the proof.

\subsection{Supplement: Rings With Infinitely Many Maximal Ideals}
\textbf{} \\ \\ \noindent
  Let us briefly consider the case of an arbitrary commutative ring.  Though others have done so (see e.g. \cite{AVL96}), it is beyond our ambitions to pursue a factorization theory in the presence of zero divisors.  But we can still ask for criteria under which there are infinitely many maximal ideals.  In this more general context $J(R) = (0)$ is no longer sufficient: e.g. $J(\C \times \C) = 0$ and there are only two maximal ideals.  Nevertheless both Euclid and Jacobson have a role to play.

\begin{prop}
\label{3.3} \cite[Prop. 4.15]{Clark-CA}
Let $I$ be an ideal of $R$ contained in the Jacobson radical.  Then for all $x \in R$, if the image of $x$ in $R/I$ is a unit, then $x$ is a unit.  In particular the natural map $R^{\times} \ra (R/I)^{\times}$ is surjective.
\end{prop}
\begin{proof}
If the image of $x$ in $R/I$ is a unit, then there is $y \in R$ such that $xy \equiv 1 \pmod{I}$, i.e.,
$xy-1 \in I \subset J(R)$.  Thus for every maximal ideal $\mm$ of $R$, $xy-1 \in \mm$ so we cannot have $x \in \mm$.  So $x$ lies in no maximal ideal of $R$ and thus $x \in R^{\times}$.
\end{proof}
%\noindent
%In general, if $I$ is an ideal of a ring $R$, the %induced homomorphism $R^{\times} \ra (R/I)^{\times}$ %need not be be surjective.  Indeed, take $R = \Z$ and %$I = n\Z$.  Then $\# \Z^{\times} = 2$ whereas $\# %(\Z/n\Z)^{\times} = \varphi(n)$, which is greater %than $2$ unless $n \in \{1,2,3,4,6\}$ (and in fact %approaches infinity with $n$).  What's going on?
%\\ \indent
%The maximal ideals of $\Z$ are precisely the ideals %of the form $(p)$ for a prime number $p$.  So the %Jacobson radical of $\Z$ consists of all integers %which are divisible by every prime number.  The only %integer which is divisible by every prime number is %$0$....because there are infinitely many primes!  %So:

\begin{thm}(Dubuque \cite{Dubuque10})
\label{3.12}
Let $R$ be an infinite ring.  If $\# R > \# R^{\times}$, then $\MaxSpec R$ is infinite.
\end{thm}
\begin{proof}
We will show by induction on $n$ that for all $n \in \Z^+$, $R$ has $n$ maximal ideals. \\
Base Case: Since $R$ is infinite, it is nonzero and thus it has a maximal ideal $\mm_1$.  \\
Induction Step: Let $\mm_1,\ldots,\mm_m$ be maximal ideals, and put
\[ I = \prod_{i=1}^m \mm_i. \]
Case 1: Suppose $I+1 \subset R^{\times}$.  Then $\# I \leq \# R^{\times}$.  Moreover
$I \subset J(R)$, so by Proposition \ref{3.3} $R^{\times} \ra (R/I)^{\times}$  is surjective.
It follows
that $\# (R/I)^{\times} \leq \# R^{\times} < \# R$: by the Chinese Remainder Theorem, $R/I \cong \prod_{i=1}^n R/\mm_i$, hence there is an injection $(R/\mm_i)^{\times} \ra (R/I)^{\times}$.   Putting the last two sentences together we conclude $\# (R/\mm_i)^{\times} < \# R$, and thus, since $R/\mm_i$ is a field and $R$ is infinite, $\# R/\mm_i = \# (R/\mm_i)^{\times} + 1 < \# R$.  Finally
this gives the contradiction
\[ \# R = \# I \cdot \# R/I = \# I \cdot \prod_{i=1}^n \# R/\mm_i < (\# R)^{n+1} = \# R. \]
Case 2: So there is $x \in I+1 \setminus R^{\times}$. Let $\mm_{n+1}$ be a maximal ideal containing $x$.  For all $1 \leq i \leq n$ we have $x-1 \in I \subset \mm_i$, so
\[ 1 = x +(1-x) \in \mm_{n+1} + \mm_i. \]
 So $\mm_{n+1}$ is an $(n+1)$st maximal ideal of $R$, completing the induction step.
\end{proof}
\noindent
A special case of Theorem \ref{3.12} appears in \cite[$\S$ 1.1, Exc. 8]{Kaplansky70}. 
 \\ \\
For a ring $R$, consider the quotient $R/J(R)$.  The maximal ideals of $R/J(R)$ correspond to the maximal ideals of $R$ containing $J(R)$ -- that is, to the maximal ideals of $R$.  Thus $R/J(R)$ is semiprimitive.  Thus we can replace any ring with a semiprimitive ring without changing its $\MaxSpec$.  However this ``Jacobson semisimplification'' need not carry domains to domains: e.g. if $R$ is a domain with $2 \leq n < \aleph_0$ maximal ideals $\mm_1,\ldots,\mm_n$, then $R/J(R) \cong \prod_{i=1}^n R/\mm_i$.  Here is a generalization.

%\begin{remark}
%\label{3.11}
%a) A special case appears in \cite[$\S$ 1.1, Exc. %8]{Kaplansky70}.  \\
%b) Theorem \ref{3.12} has the most intricate proof of any of the results given here.  Have I gone too far?  Well, it is worth noting that %this appeared some years ago as an algebra qualifying exam problem at UGA! In fact this was pointed out to me in 2006 by Dino Lorenzini and %again the following year by Brian Cook, at which point I did solve it after a few days and wrote to the department faculty about it.  I got %some feedback from Roy Smith but no one knew where the problem came from.  \\
%b) (Dubuque \cite{Dubuque10}) ``$R^{\times}$ is %finite'' can be weakened to ``$\# R^{\times} < \# %R$''.
%%, \cite[Thm. 4.20]{Clark-CA}.
%\end{remark}

\begin{thm}
a) For a ring $R$, the following are equivalent. \\
(i) $R$ has only finitely many maximal ideals.  \\
(ii) $R/J(R)$ is a finite product of fields. \\
(iii) $R/J(R)$ has only finitely many ideals. \\
(iv) $R/J(R)$ is Artinian (i.e., there are no infinite descending chains of ideals).  \\
b) A semiprimitive ring with finitely many maximal ideals has finitely many ideals.
%has finitely many maximal ideals iff it has finitely %many ideals.
\end{thm}
\begin{proof}
a) (i) $\implies$ (ii): If the maximal ideals of $R$ are $\mm_1,\ldots,\mm_n$, then by the Chinese Remainder Theorem \cite[Thm. 4.18]{Clark-CA} we have
\[ R/J(R) = R/\bigcap_{i=1}^n \mm_i \cong \prod_{i=1}^n R/\mm_i. \]
(ii) $\implies$ (iii) $\implies$ (iv) immediately.
(iv) $\implies$ (i): Maximal ideals of $R/J(R)$ correspond bijectively to maximal ideals of $R$.  And an Artinian ring has only finitely many maximal ideals \cite[Thm. 8.31]{Clark-CA}. b) This follows from part a).
\end{proof}

\section{But What About Primes?}
\noindent
Our take on Euclid's argument has been as a criterion for the existence of \emph{irreducibles}.  The distinction evaporates in a UFD.  A PID with only finitely many prime ideals is a UFD with only finitely many principal prime ideals.  It turns out that the converse is also true.\footnote{Theorem \ref{5.2} is known to the experts: see e.g. \cite{Zafrullah08}.}

\begin{thm}
\label{5.2}
Let $R$ be a UFD, not a field, with only finitely many atoms.  Then $R$ is a PID with finitely many prime ideals and $\# R = \# R^{\times}$.
\end{thm}
\begin{proof}
A UFD with finitely many nonassociate prime elements is a Cohen-Kaplansky domain, so $\MaxSpec R$ is finite and $\# R = \# R^{\times}$ by Theorem \ref{GOLOMB}.  By Theorem \ref{3.14} every nonzero prime ideal of $R$ is maximal.  The proof of Theorem \ref{3.14}a) shows: every nonzero prime ideal $\pp$ contains a prime element $p$.  Since $(p)$ is maximal, we have $\pp = (p)$. Thus every prime ideal is principal, so $R$ is a PID \cite[Thm. 4.25]{Clark-CA}.  (This is another case of the Lam-Reyes Prime Ideal Principle.) \end{proof}
\noindent
  Let us now move away from UFDs.  From Example \ref{4.15}, we deduce:

\begin{thm}
Let $\kappa \geq \aleph_0$ be a cardinal.  There is a Noetherian domain $R$ with exactly one nonzero prime ideal, exactly $\kappa$ irreducibles and no prime elements.
\end{thm}
\begin{proof}
Let $k$ be a field of cardinality $\kappa$, e.g. $k = \Q(\{t_{\alpha} \mid \alpha \in \kappa\})$. By Example \ref{4.15}, $R = k[[t^2,t^3]]$ is a Noetherian domain with one nonzero prime ideal $\mm = ( t^2,t^3 )$ and $2\kappa = \kappa$ irreducibles.  Since $\mm$ is not principal, $R$ has no prime elements.
\end{proof}
\noindent
Cohen-Kaplansky showed that an atomic domain that is neither a field nor a UFD must have at least $3$ atoms \cite[p. 469]{Cohen-Kaplansky46}.  Their argument is a nice one: we must have at least one nonprime irreducible $f_1$.  Since $(f_1)$ is not prime, it is properly contained in some prime ideal $\pp$, which must therefore contain a nonassociate irreducible $f_2$.  Since $f_1 +f_2 \in \pp$, $f_1+f_2$ is not a unit and therefore it is divisible by an irreducible $f_3$, which cannot be associate to either $f_1$ or $f_2$.
\\ \\
%\begin{remark}
%S. Gosavi, P. Pollack and I have shown, building on work %of Coykendall-Spicer \cite{Coykendall-Spicer12}, that for %all $3 \leq n < \aleph_0$ there is a Cohen-Kaplansky %domain with exactly $n$ irreducibles and no prime %elements.  This will appear elsewhere.
% \end{remark}
\noindent
Finally, we consider Dedekind domains.

\begin{ques}
\label{CLABORNQUES}
Let $R$ be a Dedekind domain with infinitely many prime ideals.  Must $R$ have infinitely many atoms?
\end{ques}
\noindent
In an important classical case the answer is \textbf{yes}, as most number theorists know.

\begin{thm}
\label{5.1}
For each number field $K$, the ring of integers $\Z_K$ has infinitely many nonassociate prime elements.
\end{thm}
\begin{proof}
Step 1: For any number field $L$, the number of rational primes that split completely in $L$ is infinite.  This is a special case of the Chebotarev Density Theorem, which however can be proved in a more elementary way, as was shown in \cite{Poonen10}.  Using some basic algebraic number theory which we omit here, it comes down to showing that for every nonconstant polynomial $f \in \Z[t]$, the set of prime numbers $p$ dividing $f(n)$ for some $n \in \Z$ is infinite.  If $f(0) = 0$ this is trivial.  If $f(0) \neq 0$, let $p_1,\ldots,p_k$ be the prime divisors of $f(0)$ (we allow $k = 0$) and let $q_1,\ldots,q_{\ell}$ be any finite set of primes not dividing $f(0)$.  For $1 \leq i \leq k$, let $a_i$ be such that $p_i^{a_i} \mid f(0)$ and $p_i^{a_i+1} \nmid f(0)$.  For $N \in \Z^+$ consider \[x_N = f(N p_1^{a_1+1} \cdots p_k^{a_k+1} q_1 \cdots q_{\ell}). \]  Then for all $1 \leq i \leq k$, $p_i^{a_i+1} \nmid x_N$ and for all $1 \leq j \leq \ell$, $q_j \nmid x_N$, so the set of $N$ for which $x_N$ is not divisible by some prime other than $p_1,\ldots,p_k,q_1,\ldots,q_{\ell}$ is finite. \\
Step 2: A prime ideal $\pp$ of a number field is principal iff it splits completely in the Hilbert class field $K^1$ of $K$.  So every prime ideal $\pp$ of $K$ lying above any one of the infinitely many prime numbers $p$ that split completely in $K^1$ is principal.
\end{proof}
\noindent
Looking at the above argument, one wonders: were we working working too hard?   Perhaps some simple argument gives a general affirmative answer to Question \ref{CLABORNQUES}.
\\ \indent
In fact Question \ref{CLABORNQUES} was answered negatively by Claborn \cite[Example 1.5]{Claborn65}.  The construction is impressively direct: start with a Dedekind domain $A$ that is not a PID, let $\mathcal{P}$ be the set of prime elements of $R$ and pass to $R = A[\{\frac{1}{p}\}_{p \in \mathcal{P}}]$.  The prime ideals of $R$ are precisely the nonprincipal prime ideals of $A$, which remain nonprincipal in $R$!   This \textbf{prime-killing construction} also appears in a work of Samuel \cite[p. 17, Thm. 6.3]{Samuel64} and is therein attributed to Nagata (cf. \cite[Lemma 2]{Nagata57}).   For a Dedekind domain $A$, write $\operatorname{Cl} A$ for its ideal class group: the quotient of the monoid of nonzero ideals of $A$ under the equivalence relation $I \sim J$ iff there are $\alpha,\beta \in A^{\bullet}$ with $(\alpha)I = (\beta)J$.  In the setting of the prime-killing construction -- i.e., $R$ is the localization of $A$ at the multiplicative subset generated by the prime elements -- we have \cite{Samuel64}, \cite{Claborn65} that $\operatorname{Cl} R \cong \operatorname{Cl} A$.  
%(Our use of the ideal class group is confined to the observation that if $R$ is an infinite Dedekind domain with $\# R = \# \operatorname{Cl} R$ then also $\# \MaxSpec R = \# R$.)

\begin{thm}
Let $\kappa$ be an infinite cardinal.  There is a Dedekind domain $R$ with exactly $\kappa$ atoms and no prime elements.
\end{thm}
\begin{proof}
We will use some properties of ``elliptic Dedekind domains'': for more details, see \cite[$\S$2.4]{Clark09}.
 Let $k$ be an algebraically closed field of characteristic $0$ and cardinality $\kappa$, and put $R = k[x,y]/(y^2-x^3-x)$.  Then $R$ is a Dedekind domain, and by the Nullstellensatz the nonzero prime ideals of $R$ are all of the form $\pp_{(x_0,y_0)} = ( x-x_0,y-y_0, y^2-x^3-x )$ for pairs $(x_0,y_0) \in k^2$ such that $y_0^2 = x_0^3 + x_0$.  In other words, they are the $k$-rational points on the projective elliptic curve $E: y^2z = x^3 + xz^2$, excluding the point at infinity $O = [0:1:0]$.  Moreover, by the Riemann-Roch Theorem, since $[x_0:y_0:1] \neq O$, the prime ideal $\pp_{(x_0,y_0)}$ is not principal.  Thus $R$ is a Dedekind domain with $\# \MaxSpec R = \# R = \kappa$ and without prime elements.
Because $R$ is Dedekind, every ideal can be generated by two elements \cite[Thm. 20.12]{Clark-CA}.  This, together with the fact that Dedekind domains are atomic domains, implies that for all $\pp \in \MaxSpec R$ there are irreducibles $p_{\pp},q_{\pp}$ such that $\pp = ( p_{\pp}, q_{\pp} )$.  Thus if $\lambda$ is the number of irreducibles of $R$ we have
\[ \kappa = \# \MaxSpec R  \leq \lambda^2 \leq (\# R)^2 = \kappa^2 = \kappa, \]
so $\lambda^2 = \kappa$.  Since $\kappa$ is infinite, so is $\lambda$ and thus $\lambda = \lambda^2 = \kappa$.
\end{proof}

\end{document}